\documentclass{amsart}

\usepackage{amsthm}
\usepackage[backend=biber, maxbibnames=99, doi=false, isbn=false, url=false]{biblatex}
\usepackage{booktabs}
\usepackage{mathtools}
\usepackage{multirow}
\usepackage{siunitx}
\usepackage{xcolor}
\usepackage{todonotes}
\usepackage{verbatim}
\usepackage{threeparttable}
\usepackage{hyperref}
\usepackage{cleveref}

\theoremstyle{plain}
\newtheorem{theorem}{Theorem}[section]

\newtheorem{lemma}[theorem]{Lemma}

\theoremstyle{definition}

\newtheorem{algorithm}[theorem]{Algorithm}

\theoremstyle{remark}

\newtheorem{remark}[theorem]{Remark}

\NewDocumentCommand{\F}{}{\mathbb{F}}
\NewDocumentCommand{\LL}{}{\mathcal{L}}
\NewDocumentCommand{\OO}{}{\mathcal{O}}
\RenewDocumentCommand{\P}{}{\mathbb{P}}
\NewDocumentCommand{\Q}{}{\mathbb{Q}}
\NewDocumentCommand{\Z}{}{\mathbb{Z}}

\DeclareMathOperator{\Cl}{Cl}
\let\div\relax
\DeclareMathOperator{\div}{div}
\DeclareMathOperator{\Div}{Div}
\DeclareMathOperator{\Frac}{Frac}

\DeclareMathOperator{\rank}{rank}
\DeclareMathOperator{\Supp}{Supp}

\newcolumntype{N}[1]{S[table-format=#1, table-align-text-before=false]}
\newcolumntype{D}[1]{S[table-format=#1, round-mode=places, round-precision=1, round-pad=false, table-align-text-after=false]}

\title[Computing class groups and gonalities of algebraic curves]{Computing class groups and gonalities of algebraic curves over finite fields}
\author{Maarten Derickx and Kenji Terao}

\begin{document}

\begin{abstract}
We give practical algorithms for computing the divisor class group and the gonality of a curve over a finite field, achieving several orders of magnitude speedup over existing methods for sufficiently large genus or residue field. The approach relies on introducing a precomputation step involving power series-expansions, which allows for an efficient amortized computation of large numbers of Riemann-Roch spaces.
\end{abstract}

    \maketitle
	\section{Introduction} \label{sec:introduction}
    The divisor class group\footnote{In the setting where our algorithms are applicable, the divisor class group is isomorphic to the Picard group.} and the gonality are two important invariants of a curve. Being able to explicitly compute these invariants over finite fields has many applications. For example, explicit computation of gonalities over finite fields has been used to study the degree $d$ points on modular curves in \cite{derickx2014gonality, derickxsutherland2017, najman2024gonality}. Efficiently computing gonalities over finite fields is a bottleneck in extending those results to higher degrees.
    
    The ability to explicitly compute class groups of curves over finite fields in Magma has seen even wider applications to current number theory research. For instance, such computations are necessary to perform the Mordell-Weil sieve \cite{bruinstoll10mwsieve}, which is an essential tool for studying rational points on (symmetric powers of) curves. Another application is to the study of torsion subgroups of Jacobians of curves over number fields. These class group computations can often take several hours, as can be seen, for example, at the end of the proof of \cite[Prop. 3.4]{dkss2023}. In prior work, the first author often had to resort to finding other proof strategies since class group computations did not terminate after several weeks. This highlights the importance of having explicit implementations of these algorithms that are not only fast in theory but also fast in practice, and is the main motivation for this article.
    
    In this paper, we present algorithms that compute the class group and gonality of a curve over a finite field. Such algorithms have been developed previously in the literature. For instance, a strategy for computing the gonality can already be found in \cite[\S 3]{derickx2014gonality}. Similarly, algorithms for the computation of the structure of the class group can be found in \cite{enge2011discretelog,hess2003classgroups,hess1999phdthesis}. The main contribution of this paper consists of implementations of our new algorithms that are up to several orders of magnitude faster in practice than existing implementations; see Section~\ref{sec:timings} for precise timing data. This substantial improvement enables many new computations that were not practical previously. For example, the gonality algorithm in this paper will be used to compute lower bounds on the $\Q$-gonality for the millions of modular curves currently being added to the LMFDB \cite{lmfdb} in forthcoming work \cite{modularcurves-wip}. Additionally, Filip Najman \cite{najman2026torsion} has already used this same algorithm to extend the results of \cite{derickxsutherland2017} to degree $7$.

    \subsection{Computing Riemann-Roch spaces}
    
    Existing algorithms for computing class groups and gonalities require computing Riemann-Roch spaces for thousands, if not millions, of divisors. So, this computation is a major bottleneck in these algorithms. Existing methods for computing such Riemann-Roch spaces, such as the one proposed by Hess in \cite{hess2002riemannroch}, use a combination of ideal and polynomial arithmetic. This is quite fast for computing a single Riemann-Roch space. However, it becomes infeasible to do this computation millions of times on higher genus curves.

    In Section \ref{sec:rr_spaces}, we present a new approach to computing many related Riemann-Roch spaces. Rather than computing each Riemann-Roch space separately, this approach precomputes the Riemann-Roch space of a single large divisor as well as series expansions for a basis of this space. While this precomputation step can be quite time-consuming, as it also relies on ideal and polynomial arithmetic, subsequent Riemann-Roch calculations are then reduced to linear algebraic problems, which can be solved extremely quickly. This provides a very efficient amortized algorithm for computing many Riemann-Roch spaces at once.

    Applying this technique to existing algorithms for computing gonalities and class groups of curves forms the main idea behind our new algorithms and subsequent implementations. For gonalities, this is mostly straightforward, simply requiring an application of the Riemann-Roch theorem. The resulting algorithm is described in Section \ref{sec:gonalities}. In the case of class groups, the necessary modifications are somewhat more involved, and are described in Section \ref{sec:class_groups}.

    \subsection{Data availability and reproducibility.}
    All the code and data for this article are available at \url{https://github.com/nt-lib/CurveArith}. Our code depends on Magma \cite{magma} and has been verified to work with versions v2.29-1 and v2.29-4.

    The timing data of Section \ref{sec:timings} has been obtained by running our code on a server with an AMD Opteron 6380 CPU @ 2.30GHz with 32 cores (64 threads) and 128GB of RAM running Ubuntu 22.04.5 LTS and Magma v2.29-1. 

    The \texttt{README.md} file in the GitHub repository contains more detailed instructions on how to use the code. In addition, it contains instructions on how to recreate the timing data used in Section~\ref{sec:timings} and on how to run the automated tests which verify that the code is working as intended.

    In Table \ref{tab:gonality_x0} we reproduce the results from \cite[Proposition 5.15]{najman2024gonality} on the gonality lower bounds for $X_0(N)$ over $\F_p$, with a few caveats. First of all, the table in the aforementioned proposition gives the prime $p=6$ for $N=132, 148, 277$; we follow the accompanying code instead and use the prime 5. Similarly, the table states a lower bound of $8$ on the $\F_5$-gonality of $X_0(154)$ while the accompanying code computes a lower bound on the $\F_3$-gonality. We verified that both gonalities are at least 8.
    
    We also note that the timing data in Table \ref{tab:gonality_x0} is not directly comparable to the timing data from \cite{najman2024gonality}, due to various reasons. Firstly, our computations were done on a machine with a lower single thread performance according to \url{cpubenchmark.net/single-thread/server}. Moreover, the data of \cite{najman2024gonality} includes the time taken to compute a model for $X_0(N)$, which we have omitted. The models and sets of divisors used are also different in both cases; see \cite[Propositions 5.13, 5.14]{najman2024gonality}. Nevertheless, we found that their overall computation time of 860 hours is relatively close to the time of 1240 hours we estimate it would take to reproduce their computations with Algorithm \ref{alg:has_function_original}. Our reproduction being slower is in line with the cpubenchmark. The reproduction of the gonality computation with our new Algorithm \ref{alg:has_function_new} took 10 hours, showing that we achieve a speedup of roughly two orders of magnitude in this concrete application.
    
    \subsection{Acknowledgements}

    The gonality part of this article grew out of an ICERM project suggested by Andrew V. Sutherland at the Algebraic Points on Curves workshop in June 2025. We would like to thank the participants of this group, Abbey Bourdon, Renee Bell, Sameera Vemulapalli and Travis Morrison, for their collaboration at ICERM. The second author thanks Filip Najman for their generous hospitality during the author's stay at the University of Zagreb, during which portions of this work were carried out. The authors also thank the referees for their many helpful comments on both this manuscript and the associated \texttt{CurveArith} repository.
    
    The first author was supported by the project ``Implementation of cutting-edge research
    and its application as part of the Scientific Center of Excellence for Quantum and Complex Systems, and Representations of Lie Algebras'', PK.1.1.10.0004, European Union,
    European Regional Development Fund and by the Croatian Science Foundation under the
    project no. IP-2022-10-5008.
    
    The second author was supported by the Additional Funding Programme for Mathematical Sciences, delivered by EPSRC (EP/V521917/1) and the Heilbronn Institute for Mathematical Research.
    
	\section{Notation} \label{sec:notation}

    We fix the following notation which will persist throughout the paper. Let $X$ be a smooth, projective, geometrically integral curve over a finite field $k$. We denote by $k(X) = \OO_{X, \xi}$ the function field of $X$, where $\xi$ is the generic point of $X$.

    Let $x \in X$ be a closed point of $X$. We denote by $\mathfrak{m}_x$ the maximal ideal of the DVR $\OO_{X, x}$, and $k(x) = \OO_{X, x} / \mathfrak{m}_x$ the residue field of $x$. We fix a uniformizer $q_x \in \OO_{X, x}$ at $x$, that is to say, $q_x$ generates the ideal $\mathfrak{m}_x$. There is a canonical isomorphism $\Frac(\OO_{X, x}) \cong k(X)$, and we denote by $v_x : k(X)^{\times} \to \Z$ the valuation induced by the standard valuation on $\OO_{X, x}$.
    
    We view $\OO_{X, x}$ as a subring of its completion $\widehat{\OO}_{X, x}$. Since the extension $k(x) / k$ is a finite separable extension, Hensel's lemma yields a subfield of $\widehat{\OO}_{X, x}$ isomorphic to $k(x)$. The map $q \mapsto q_x$ gives an isomorphism $\widehat{\OO}_{X, x} \cong k(x)[[q]]$ of $k(x)$-algebras. More generally, we obtain an embedding $k(X) \cong \Frac(\OO_{X, x}) \hookrightarrow \Frac(\widehat{\OO}_{X, x}) \cong k(x)(\!(q)\!)$. We call the image of $f \in k(X)$ under this map the Laurent series expansion or $q$-expansion of $f$ at $x$.

    We denote by $\Div X$ the group of Weil divisors of $X$, and by $\Cl X$ the divisor class group of $X$. The latter is a finitely generated abelian group of rank 1, whose free part is generated by the class of a degree 1 divisor on $X$.

    Throughout the paper, we shall assume that we are able to do simple arithmetic in the function field $k(X)$, compute a basis of the space $\Omega^1_X$ of holomorphic differentials on $X$ and enumerate the closed points of $X$ of a given degree $d$. Given a closed point $x \in X$, we will also assume that we are able to compute a $k$-basis of the residue field $k(x)$, evaluate the quotient map $\OO_{X, x} \to k(x)$ at functions $f \in \OO_{X, x}$ and compute the valuation $v_x(f)$ for any function $f \in k(X)^\times$.

    In practice, this functionality is provided by existing intrinsics in Magma \cite{magma}, which are in turn based on the work of Hess in \cite{hess2002riemannroch} and the references therein. Briefly, the function field $k(X)$ is represented as a finite extension $k(t)[\alpha]$ of the rational function field $k(t)$. This is equivalent to equipping the curve $X$ with a fixed map $\pi : X \to \P^1_k$. We denote by $\OO_0$ the maximal finite order of $k(X)$, that is to say, the integral closure of $k[t]$ inside of $k(X)$. The closed points of $X$ not lying above $\infty \in \P^1_k$ correspond to the maximal ideals of $\OO_0$, while those lying above $\infty$ are represented by maximal ideals in the similarly defined maximal infinite order of $k(X)$. The operations described in the previous paragraph are then provided in Magma by the \texttt{BasisOfHolomorphicDifferentials}, \texttt{Places}, \texttt{ResidueClassField}, \texttt{Evaluate} and \texttt{Valuation} intrinsics.

    While these existing intrinsics provide a practical and convenient base for the development of our algorithms, they may not yield the best performance, especially in certain edge cases; see Remark \ref{rmk:x0_198}. It would be a worthwhile effort to develop a representation of curves and points that is especially suited to our algorithms, as we suspect this might lead to further performance gains.

	\section{Computing Riemann-Roch spaces} \label{sec:rr_spaces}
	
	Let $D \in \Div X$ be a divisor on $X$. The Riemann-Roch space $\LL(D)$ of $D$ is defined to be
	\[
		\LL(D) = \{0\} \cup \{f \in k(X)^{\times} : v_x(f) \geq -v_x(D) \quad \forall x \in X\}.
	\]
	This is a $k$-vector space, and we denote by $\ell(D) = \dim_k \LL(D)$ its dimension. In this section, we present a new approach to computing many related Riemann-Roch spaces at once.
	
	The basis of this algorithm stems from the following observation. Consider a set of functions $\{f_1, \dots, f_n\} \in k(X)$ and a closed point $x \in X$. Let
	\[
		f_i = \sum_{j \in \Z} a_{i, j} q^j \in k(x) (\!(q)\!)
	\]
	be the Laurent series expansion of the function $f_i$ at $x$, for all $1 \leq i \leq n$. Given a linear combination $f = b_1 f_1 + \dots + b_n f_n$, with $b_i \in k$, the Laurent series expansion of $f$ at $x$ is easily computed; we have
	\[
		f = \sum_{j \in \Z} \left(\sum_{i = 1}^{n} a_{i, j} b_i\right) q^j.
	\]
	The valuation of the function $f$ at $x$ can now be read off from this Laurent series expansion: if $v_x(f_i) \geq m$ for all $1 \leq i \leq n$, then $v_x(f) \geq m + l$ if and only if
	\[
		\sum_{i = 1}^{n} a_{i, m + j} b_i = 0
	\]
	for all $0 \leq j < l$. This is a linear algebraic condition; it is equivalent to the vector $(b_1, \dots, b_n) \in k(x)^n$ being contained in the nullspace of the matrix
	\[
		\begin{bmatrix}
			a_{1, m} & \cdots & a_{1, m + l - 1} \\
			\vdots & & \vdots \\
			a_{n, m} & \cdots & a_{n, m + l - 1}
		\end{bmatrix}.
	\]
	If we take the set of functions $\{f_1, \dots, f_n\}$ to form a basis of some Riemann-Roch space $\LL(D_0)$, we have thus shown that the Riemann-Roch space $\LL(D_0 - l x)$ can be identified with the nullspace of the above matrix.
	
	The following theorem, which gives a more technical version of the above argument, forms the basis of our algorithm for computing Riemann-Roch spaces.
	
	\begin{theorem} \label{thm:rr_matrix}
		Let $D_0 \in \Div X$ be a divisor on X, and let $\{f_1, \dots, f_{\ell}\}$ be a $k$-basis of $\LL(D_0)$, where $\ell = \ell(D_0)$. Note that this choice of basis defines an isomorphism $\LL(D_0) \cong k^{\ell}$ of $k$-vector spaces.
		
		For any closed point $x \in X$, fix a function $g_x \in k(X)^\times$ and a set of functions $\{u_{x, 1}, \dots, u_{x, \deg(x)}\} \subset \widehat{\OO}_{X, x}$ whose residue classes form a $k$-basis of $k(x)$. For any $i \in \{1, \dots, \ell\}$, consider the unique expression
		\[
			f_i g_x = \sum_{j = 0}^{\infty} \left(a_{x, i, j, 1} u_{x, 1} + \dots + a_{x, i, j, \deg(x)} u_{x, \deg(x)}\right) q_x^{j - v_x(D_0) + v_x(g_x)},
		\]
		where all coefficients $a_{x, i, j, l}$ lie in $k$. For all $i, j$, we let $\mathbf{a}_{x, i, j} \in k^{\deg(x)}$ be the vector
		\[
			\mathbf{a}_{x, i, j} = (a_{x, i, j, 1}, \dots, a_{x, i, j, \deg(x)}).
		\]
		
		Let $D = m_{x_1} x_1 + \dots + m_{x_n} x_n \in \Div X$ be an effective divisor on $X$, where the points $x_i$ are distinct. Then the isomorphism $\LL(D_0) \cong k^{\ell}$ identifies the subspace $\LL(D_0 - D)$ with the nullspace of the $\ell \times \deg(D)$ matrix
		\[
			A = \begin{bmatrix}
			\mathbf{a}_{x_1, 1, 0} & \cdots & \mathbf{a}_{x_1, 1, m_{x_1} - 1} & \cdots & \mathbf{a}_{x_n, 1, 0} & \cdots & \mathbf{a}_{x_n, 1, m_{x_n} - 1} \\
			\vdots & & \vdots & & \vdots & & \vdots \\
			\mathbf{a}_{x_1, \ell, 0} & \cdots & \mathbf{a}_{x_1, \ell, m_{x_1} - 1} & \cdots & \mathbf{a}_{x_n, \ell, 0} & \cdots & \mathbf{a}_{x_n, \ell, m_{x_n} - 1}
		\end{bmatrix}.
		\]
		In particular, we have
		$
			\ell(D_0 - D) = \ell(D_0) - \rank A.
		$
	\end{theorem}
	
	\begin{proof}
		Let $f = b_1 f_1 + \dots + b_\ell f_\ell$ be a non-zero element of $\LL(D_0)$, with $b_i \in k$. By definition, we have that $f \in \LL(D_0 - D)$ if and only if
		$
			v_x(f) \geq v_x(D) - v_x(D_0)
		$
		for all closed points $x \in \{x_1, \dots, x_n\}$. Fix such a closed point $x$. We have
		\[
			v_x(f) \geq v_x(D) - v_x(D_0) \iff v_x(f g_x) \geq v_x(D) - v_x(D_0) + v_x(g_x).
		\]
		Moreover, we can write the product $f g_x$ as
		\[
			f g_x = \sum_{i = 1}^{\ell} b_i f_i g_x = \sum_{j = 0}^{\infty} \left( \sum_{l = 1}^{\deg(x)} \left( \sum_{i = 1}^{\ell} b_i a_{x, i, j, l} \right) u_{x, l} \right) q_x^{j - v_x(D_0) + v_x(g_x)}.
		\]
		Since $q_x$ is a uniformizer at $x$, it follows that $v_x(f g_x) \geq 1 - v_x(D_0) + v_x(g_x)$ if and only if
		\[
			v_x \! \left( \sum_{l = 1}^{\deg(x)} \left( \sum_{i = 1}^{\ell} b_i a_{x, i, 0, l} \right) u_{x, l} \right) > 0.
		\]
		As the residue classes of the functions $u_{x, l}$ form a $k$-basis of $k(x)$, this latter condition holds if and only if
		\[
			\sum_{i = 1}^{\ell} b_i a_{x, i, 0, l} = 0
		\]
		for all $1 \leq l \leq \deg(x)$. In this case, we can then write the product $f g_x$ as
		\[
			f g_x = \sum_{j = 1}^{\infty} \left( \sum_{l = 1}^{\deg(x)} \left( \sum_{i = 1}^{\ell} b_i a_{x, i, j, l} \right) u_{x, l} \right) q_x^{j - v_x(D_0) + v_x(g_x)}.
		\]
		Proceeding inductively, we find that $v_x(f g_x) \geq v_x(D) - v_x(D_0) + v_x(g_x)$ if and only if
		\[
			\sum_{i = 1}^{\ell} b_i a_{x, i, j, l} = 0
		\]
		for all $0 \leq j < v_x(D)$ and all $1 \leq l \leq \deg(x)$. The statement of the theorem now follows by the construction of the vectors $\mathbf{a}_{x, i, j}$.
	\end{proof}
	
	In addition to generalizing the aforementioned observation to arbitrary divisors $D$, this result incorporates two main differences with the former. Firstly, the functions $f_i$ can be multiplied by an arbitrary function $g_x \in k(X)^\times$. As will be described below, this will allow us to normalize the valuations of the functions $f_i$ without needing to compute large powers of the uniformizing parameter $q_x$.
	
	Secondly, rather than work with the Laurent series expansion of the functions $f_i$, we work with a more general series expansion involving the functions $u_{x, l}$. Taking the standard Laurent series expansion corresponds to choosing the functions $u_{x, l}$ to be a $k$-basis of the subring $k(x) \subset \widehat{\OO}_{X, x}$. Working with this more general formulation will allow us to avoid computing the subring $k(x) \subset \widehat{\OO}_{X, x}$, which is usually done via Hensel lifting and can be quite time-consuming and cumbersome.
	
	This result provides a very efficient way of computing many Riemann-Roch spaces of the form $\LL(D_0 - D)$, where $D_0$ is a fixed divisor on $X$, and $D$ is an effective divisor on $X$ with bounded degree. Indeed, given a basis of the Riemann-Roch space $\LL(D_0)$, which can be efficiently obtained using the method of Hess \cite{hess2002riemannroch}, we can precompute and cache the necessary vectors $\mathbf{a}_{x, i, j}$, as these depend only on the fixed divisor $D_0$ and the fixed choice of the functions $g_x$ and $u_{x, l}$. Computing the Riemann-Roch space $\LL(D_0 - D)$ then reduces to computing the nullspace of the matrix $A$ appearing in Theorem \ref{thm:rr_matrix}, which can be done extremely quickly.
	
	The only remaining question is how to compute the vectors $\mathbf{a}_{x, i, j}$ efficiently. Before tackling this however, we first describe a new procedure for computing a uniformizer $q_x$ at a closed point $x \in X$, which may be of independent interest. While algorithms for computing such uniformizers already exist, such as the \texttt{UniformizingParameter} intrinsic in Magma, we have found that the uniformizers returned by such algorithms are often comprised of polynomials of large degree. This tends to bog down algorithms which use such uniformizers significantly, as polynomial arithmetic becomes much slower.

    Instead, we propose the following scheme based on coordinate functions on a model of $X$.
   
    \begin{algorithm} \label{alg:uniformizer}
		(Uniformizing parameter)
		
		Input: A curve $X \subseteq \P^n_k$ over a perfect field $k$ and a closed point $x \in X$ at which $X$ is smooth.
		
		Output: A function $f \in k(X)$ such that $v_x(f) = 1$.
		
		\begin{enumerate}
			\item Find a coordinate $y \in \{y_0, \dots, y_n\}$ such that $y(x) \neq 0$.
			\item For each $0 \leq i \leq n$ such that $y_i \neq y$:
			\begin{enumerate}
				\item Compute $a := \frac{y_i}{y}(x) \in k(x)$.
				\item Compute the minimal polynomial $\mu_a \in k[t]$ of $a$ over $k$.
				\item Evaluate $f := \mu_a(\frac{y_i}{y}) \in k(X)$.
				\item If $v_x(f) = 1$, return $f$ and terminate.
			\end{enumerate}
		\end{enumerate}
	\end{algorithm}

    The advantage of the above algorithm over the one currently in Magma is that the uniformizer will always be a polynomial of degree at most $[k(x):k]$ in one of the coordinate functions. This yields a much smaller expression than the built-in Magma function, which sometimes yields screen-filling expressions even for $k$-rational points.

    The correctness of the above algorithm relies on the following fact: if $\overline{x}$ is a point of the base change $X_{\overline{k}}$ lying above $x \in X$, where $\overline{k}$ is the algebraic closure of $k$, then one of the functions $\frac{y_i}{y} - \frac{y_i}{y}(\overline{x})$ must be a uniformizer at $\overline{x}$, since $X_{\overline{k}}$ is smooth at $\overline{x}$. Since $\overline{x}$ lies above $x$, we obtain a homomorphism of local rings $\OO_{X, x} \to \OO_{X_{\overline{k}}, \overline{x}}$. Now, we have
    \[
        \mu_a\left(\frac{y_i}{y}\right) = \prod_{\sigma \in \mathrm{Gal}(k(a)/k)} \left(\frac{y_i}{y} - \sigma\!\left(\frac{y_i}{y}(\overline{x})\right)\right).
    \]
    In particular, $\mu_a(\frac{y_i}{y})$ has valuation 1 when viewed as an element of $\OO_{X_{\overline{k}}, \overline{x}}$, since $\frac{y_i}{y} - \frac{y_i}{y}(\overline{x})$ is a uniformizer, while the other terms in the product have valuation 0. Thus, $\mu_a(\frac{y_i}{y})$ is a uniformizer at $x$, as desired.

    In our code, we apply Algorithm \ref{alg:uniformizer} to either the canonical model if the curve is non-hyperelliptic or to the hyperelliptic model, as these models are always smooth\footnote{Some special care is needed for the point at infinity in the hyperelliptic case.}. Also note that while Algorithm \ref{alg:uniformizer} is stated for a curve $X \subseteq \P^n$, we do not need to compute equations for the model of $X$ in $\P^n$. Instead, Algorithm \ref{alg:uniformizer} simply requires a map $f: X \to \P^n$ that is birational onto its image and a point $x$ such that $f(X)$ is smooth at $x$. This data can be obtained in the computational setting of Section \ref{sec:notation}; see Algorithm \ref{alg:differential_expansions} for a practical example.
    
	We now turn our attention to the problem of computing the vectors $\mathbf{a}_{x, i, j}$ appearing in the statement of Theorem \ref{thm:rr_matrix}. Conceptually, this procedure is quite simple. After multiplying the function $f_i$ by an appropriate power of the uniformizer $q_x$ so that the resulting function has valuation 0, we evaluate this function at $x$, or equivalently, compute the residue class of the function in $k(x)$. This yields the vector $\mathbf{a}_{x, i, 0}$. We then subtract the corresponding linear combination of the functions $u_{x, l}$, giving a function of positive valuation, and then divide by the uniformizer $q_x$. We can now repeat this procedure to obtain the higher terms in the series expansion.
	
	As far as we are aware, this is the approach used by existing algorithms for computing standard Laurent series expansions such as the \texttt{Completion} intrinsic in Magma. However, this approach can be quite time-consuming for a variety of reasons. For instance, if the uniformizer $q_x$ is a polynomial of high degree, the division operation can be quite slow. To circumvent this, we utilize the procedure described in Algorithm \ref{alg:uniformizer} instead, which tends to yield simpler uniformizers.

	Moreover, if the valuations of the functions $f_i$ at $x$ have large absolute value, the first step of the procedure requires one to compute a large power of the uniformizer. This will inevitably be very cumbersome and lead to many of the same problems as before. As such, instead of multiplying by a power of the uniformizer, we instead divide by a known function of the correct valuation, namely the function $f_i$ with minimal valuation. This explains the necessity of including the function $g_x$ in the statement of Theorem \ref{thm:rr_matrix}.
	
	Finally, when computing the standard Laurent series expansion of the functions $f_i$, one is required to use functions $u_{x, l}$ which form a $k$-basis of the subring $k(x) \subset \widehat{\OO}_{X, x}$. These functions are obtained via Hensel lifting, which can be quite slow, and are described by polynomials of large degree. To alleviate both of these issues, we utilize different functions $u_{x, l}$ whose residue classes, as described in Theorem \ref{thm:rr_matrix}, are only required to form a $k$-basis of the residue class field $k(x)$. These functions are usually much simpler, and their construction does not require Hensel lifting. In fact, these can simply be taken to be the original inputs for the Hensel lifting procedure. In practice, these are obtained using the \texttt{ResidueClassField} and \texttt{Lift} intrinsics in Magma.
	
	Combining these three improvements, we obtain the following algorithm for computing the vectors $\mathbf{a}_{x, i, j}$ appearing in Theorem \ref{thm:rr_matrix}.
	
	\begin{algorithm} \label{alg:function_expansions}
		(Function expansions)
		
		Input: A $k$-basis $\{f_1, \dots, f_{\ell}\}$ of $\LL(D_0)$, a closed point $x \in X$ and $m \geq 1$.
		
		Output: The vectors $\mathbf{a}_{x, i, j} \in k^{\deg(x)}$ associated to the function $f_i$ as defined in Theorem \ref{thm:rr_matrix} for $1 \leq i \leq \ell$ and $0 \leq j < m$, for some functions $g_x$ and $u_{x, l}$ chosen at runtime.
		
		\begin{enumerate}
			\item Compute a set of functions $\{u_{x, 1}, \dots, u_{x, \deg(x)}\} \subseteq k(X)$ whose residue classes form a $k$-basis of $k(x)$ and the associated isomorphism $\varphi : k(x) \to k^{\deg(x)}$ of $k$-vector spaces. \vspace{0.5em}
			\item Compute $m' := \min_{1 \leq i \leq \ell} v_{x}(f_i) - v_{x}(D_0)$.
			\item Set $\mathbf{a}_{x, i, j} = 0$ for $1 \leq i \leq \ell$ and $0 \leq j < m'$.
			\item If $m' \geq m$, return the coefficients $\mathbf{a}_{x, i, j}$ for $0 \leq j < m$ and terminate. \vspace{0.5em}
			\item Find a function $f \in \{f_1, \dots, f_\ell\}$ such that $v_x(f) - v_x(D_0) = m'$. Set $g_x = \frac{1}{f}$.
			\item Set $f_i := f_i g_x$ and $\mathbf{a}_{x, i, m'} := \varphi(f_i(x)) \in k^{\deg(x)}$ for all $1 \leq i \leq \ell$.
			\item If $m = m' + 1$, return the coefficients $\mathbf{a}_{x, i, j}$ and terminate. \vspace{0.5em}
			\item Compute a uniformizing parameter $q_x$ of $x$ using Algorithm \ref{alg:uniformizer}.
			\item For each $m' < j < m$ and $1 \leq i \leq \ell$:
			\begin{enumerate}
				\item Evaluate the dot product $g_i := \mathbf{a}_{x, i, j - 1} \cdot (u_{x, 1}, \dots, u_{x, \deg(x)}) \in k(X)$.
				\item Set $f_i := \frac{f_i - g_i}{q_x}$ and $\mathbf{a}_{x, i, j} := \varphi(f_i(x))$.
			\end{enumerate}
			\item Return the coefficients $\mathbf{a}_{x, i, j}$ and terminate.
		\end{enumerate}
	\end{algorithm}
	
	We note that the computation of the uniformizing parameter $q_x$ is done in step 8, after the first $m'$ vectors are computed. In particular, if $m = 1$, the algorithm will always terminate before or at step 7, and one does not need to compute a uniformizing parameter at all. This leads to a particularly efficient procedure for computing the coefficients $\mathbf{a}_{x, i, 0}$, which will be exploited further in Section \ref{sec:class_groups}.
	
	Another particular case of interest occurs when the base divisor $D_0$ is a canonical divisor of $X$. Let $\{\omega_1, \dots, \omega_g\}$ be a $k$-basis of the space of holomorphic differentials $\Omega^1_X$ on $X$, where $g > 0$ is the genus of $X$. In this case, a basis of the Riemann-Roch space $\LL(K_X)$ of the canonical divisor $K_X = \div(\omega_1)$ is given by $\{\frac{\omega_1}{\omega_1}, \dots, \frac{\omega_g}{\omega_1}\}$. For any closed point $x \in X$, there exists a differential $\omega \in \{\omega_1, \dots, \omega_g\}$ such that $v_x(\omega) = 0$. In particular, $v_x(\frac{\omega}{\omega_1}) = - v_x(\omega_1) = - v_x(K_X)$. Thus, in Algorithm \ref{alg:function_expansions}, we obtain $m' = 0$, and steps 2 through 4 can be omitted. Moreover, the function $g_x$ in step 5 can be taken to be $\frac{\omega_1}{\omega}$, in which case the functions $f_i g_x$ of step 6 simply become the functions $\frac{\omega_1}{\omega}, \dots, \frac{\omega_g}{\omega}$. This gives the following simpler algorithm for computing the series expansions in the case of the canonical divisor.
	
	\begin{algorithm} \label{alg:differential_expansions}
		(Differential expansions)
		
		Input: A $k$-basis $\{\omega_1, \dots, \omega_{g}\}$ of $\Omega^1_X$, a closed point $x \in X$ and $m \geq 1$.
		
		Output: The vectors $\mathbf{a}_{x, i, j} \in k^{\deg(x)}$ associated to the function $\frac{\omega_i}{\omega_1}$ as defined in Theorem \ref{thm:rr_matrix} for $1 \leq i \leq g$ and $0 \leq j < m$, for some functions $g_x$ and $u_{x, l}$ chosen at runtime.
		
		\begin{enumerate}
			\item Compute a set of functions $\{u_{x, 1}, \dots, u_{x, \deg(x)}\} \subseteq k(X)$ whose residue classes form a $k$-basis of $k(x)$ and the associated isomorphism $\varphi : k(x) \to k^{\deg(x)}$ of $k$-vector spaces. \vspace{0.5em}
			\item Find a differential $\omega \in \{\omega_1, \dots, \omega_g\}$ such that $v_x(\omega) = 0$.
			\item Set $f_i := \frac{\omega_i}{\omega}$ and $\mathbf{a}_{x, i, 0} := \varphi(f_i(x)) \in k^{\deg(x)}$ for all $1 \leq i \leq \ell$.
			\item If $m = 1$, return the coefficients $\mathbf{a}_{x, i, 0}$ and terminate. \vspace{0.5em}
			\item Compute a uniformizing parameter $q_x$ of $x$ using Algorithm \ref{alg:uniformizer} with respect to the canonical embedding of $X$ in $\P^{g-1}$, with affine coordinates $\frac{y_i}{y} = f_i$.
			\item For each $0 < j < m$ and $1 \leq i \leq \ell$:
			\begin{enumerate}
				\item Evaluate the dot product $g_i := \mathbf{a}_{x, i, j - 1} \cdot (u_{x, 1}, \dots, u_{x, \deg(x)}) \in k(X)$.
				\item Set $f_i := \frac{f_i - g_i}{q_x}$ and $\mathbf{a}_{x, i, j} := \varphi(f_i(x))$.
			\end{enumerate}
			\item Return the coefficients $\mathbf{a}_{x, i, j}$ and terminate.
		\end{enumerate}
	\end{algorithm}

    Note that step 5 requires the canonical map to be an embedding; this is the case when $X$ is geometrically non-hyperelliptic. If the curve $X$ is hyperelliptic, we instead fall back to the existing \texttt{UniformizingParameter} Magma intrinsic.
	
	\section{Computing gonalities} \label{sec:gonalities}
	
	Equipped with this procedure for computing Riemann-Roch spaces of divisors on $X$, we now turn our attention to the problem of computing the gonality of $X$ over $k$. This problem has received significant attention, particularly in the case when $k$ is an algebraically closed field. For instance, Brill-Noether theory provides strong theoretical bounds on the gonality of $X$, such as the Kleiman-Laksov theorem \cite{kleiman1972existence}. Schicho, Schreyer and Weimann, in \cite{schicho2013computational}, provide an algorithm for computing gonal maps of curves $X$ defined over algebraically closed fields of arbitrary characteristic. This algorithm, based on syzygy computations, is somewhat inefficient in practice, even for curves of moderately sized genus.
    
    For curves defined over finite fields $k$, the state-of-the-art algorithms are somewhat more naive, and rely on the following observation. Let $f \in k(X)$ be a non-constant function of degree at most $d$. By definition, the pole divisor $D$ of $f$ has degree at most $d$, and $f \in \LL(D)$. In particular, enlarging the divisor $D$ if necessary, we have $\ell(D) \geq 2$ for some effective divisor $D$ of degree $d$. Conversely, if this latter condition holds, then there exists a non-constant function $f \in \LL(D)$ whose pole divisor is at most $D$. In particular, $f$ has degree at most $d$.
	
	As the curve $X$ is defined over a finite field $k$, the set of effective divisors on $X$ of degree $d$ is finite and can be enumerated. In particular, we can iterate through all the Riemann-Roch spaces $\LL(D)$ described above, and check whether any have dimension at least 2. This yields a simple procedure for determining whether the curve $X$ has a non-constant function of degree at most $d$.

	Rather than examine all of the effective divisors $D$ of degree $d$ on the curve $X$, it is possible to restrict our search to divisors of a certain shape by exploiting the structure of the curve $X$. Such techniques have previously appeared in the literature, such as in \cite[Proposition 5]{derickx2014gonality} and \cite[Propositions 5.13, 5.14]{najman2024gonality}. In our case, we apply a straightforward procedure based on the following lemma.
    
    \begin{lemma}[{\cite[Lemma 3.1]{najman2024gonality}}] \label{thm:rational_points_trick}
        Suppose that there exists a function $f \in k(X)$ of degree $d$. Then there exists a function $g \in k(X)$ of degree $d$ such that the pole divisor of $g$ is supported on at least $n_1 = \left\lceil \frac{|X(k)|}{|k| + 1} \right\rceil$ points $x \in X(k)$.
    \end{lemma}
    
    As a result, rather than enumerate all effective divisors $D \in \Div X$ of degree $d$, we can restrict to considering all such effective divisors supported on at least $n_1$ points $x \in X(k)$. Combining this with the procedure described above gives the following brute-force algorithm to determine whether the curve $X$ has a non-constant function of degree at most $d$, variations of which already occur in \cite{derickx2014gonality} and \cite{najman2024gonality}.
	
	\begin{algorithm} \label{alg:has_function_original}
		(Has function of degree $\leq d$)
		
		Input: A curve $X$ of genus $g$ over a finite field $k$ and $d \geq 1$.

		Output: Whether the curve $X$ has a function $X \to \P^1_k$ of degree at most $d$.
		
		\begin{enumerate}
            \item Set $n_1 = \left\lceil \frac{|X(k)|}{|k| + 1} \right\rceil$.
			\item Compute all closed points of $X$ of degree $\leq \max(1, d - n_1)$.
			\item For all effective divisors $D \in \Div X$ of degree $d$ and supported on at least $n_1$ points of $X(k)$:
			\begin{enumerate}
				\item Compute $\ell(D)$ using one of the already existing algorithms, for example those given in \cite{hess2002riemannroch}.
				\item If $\ell(D) \geq 2$, return true and terminate.
			\end{enumerate}
			\item Return false and terminate.
		\end{enumerate}
	\end{algorithm}
	
	This algorithm can easily be extended to compute the gonality of the curve $X$, by successively checking each degree $d$ in increasing order until we find the minimum degree $d$ such that the curve $X$ has a function of degree $d$. The gonality of the curve $X$ is then equal to this degree $d$.
	
	This algorithm, with a few modifications, was for instance used by the first author and van Hoeij in \cite{derickx2014gonality} and Najman and Orli\'c in \cite{najman2024gonality} to compute the $\F_p$-gonalities of a number of modular curves $X_1(N)$ and $X_0(N)$, respectively. We note that this algorithm has exponential complexity, as the number of effective divisors of degree $d$ on the curve $X$ grows as $q^d$, where $q$ is the order of the field $k$. This makes this algorithm infeasible for large $q$ or moderately large degrees $d$.
	
	The key bottleneck in this algorithm is the computation of the Riemann-Roch dimension $\ell(D)$ for many different divisors $D$. As such, we want to replace this step with the linear algebraic procedure offered by Theorem \ref{thm:rr_matrix}. However, we cannot apply this result directly, as the divisors for which the Riemann-Roch dimension need to be computed do not have the right shape. Instead, we make use of the Riemann-Roch theorem, which states that, given a divisor $D \in \Div X$, we have
	\[
		\ell(D) - \ell(K_X - D) = \deg(D) - g + 1,
	\]
	where $K_X$ is a canonical divisor on $X$. In particular, we find that $\ell(D) \geq 2$ if and only if $\ell(K_X - D) \geq g + 1 - \deg(D)$. This latter condition can be rewritten in the form
	$
		\deg(D) > \ell(K_X) - \ell(K_X - D).
	$
	The right hand side can be computed as the rank of a matrix $A$ as given in Theorem \ref{thm:rr_matrix}, leading to the following algorithm for computing whether the curve $X$ has a function of degree at most $d$.
	
	\begin{algorithm} \label{alg:has_function_new}
		(Has function of degree $\leq d$)
		
		Input: A curve $X$ of genus $g$ over a finite field $k$ and $d \geq 1$.
		
		Output: Whether the curve $X$ has a function $X \to \P^1_k$ of degree at most $d$.
		
		\begin{enumerate}
            \item Set $n_1 = \left\lceil \frac{|X(k)|}{|k| + 1} \right\rceil$.
            \item If $n_1 > d$ return false and terminate.
			\item Compute all closed points of $X$ of degree $\leq \max(1,d - n_1)$.
			\item Compute a $k$-basis $\{\omega_1, \dots, \omega_g\}$ of $\Omega^1_X$.
			\item For all closed points $x \in X$ with $\deg(x) \leq \max(1,d - n_1)$, compute the vectors $\mathbf{a}_{x, i, j}$ associated to the function $\frac{\omega_i}{\omega_1}$ as defined in Theorem \ref{thm:rr_matrix} for $1 \leq i \leq g$ and $0 \leq j < \left\lfloor \frac{d - n_1}{\deg(x)} \right\rfloor$ using Algorithm \ref{alg:differential_expansions}.
			\item For all effective divisors $D \in \Div X$ of degree $d$ and supported on at least $n_1$ points of $X(k)$:
			\begin{enumerate}
				\item Compute the rank of the matrix $A$ appearing in Theorem \ref{thm:rr_matrix}.
				\item If $\rank(A) < d$, return true and terminate.
			\end{enumerate}
			\item Return false and terminate.
		\end{enumerate}
	\end{algorithm}
    
	As before, this algorithm has exponential complexity, since we still iterate through most of the effective divisors $D$ of degree $d$ on the curve $X$. However, the computation of the Riemann-Roch dimension is now replaced by a single rank computation, which is much faster. This leads to a significantly faster algorithm, making it feasible to compute the gonalities of higher genera curves over larger base fields.
	
	While this section has focused on improvements to Algorithm \ref{alg:has_function_original}, there exist other approaches to computing the gonality of the curve $X$ that combine the enumeration of a large set of effective divisors of $X$ with other techniques to obtain further speedups. For example, in \cite[Section 3.2]{derickx2014gonality}, van Hoeij and the first named author determine that the modular curve $X_1(37) / \F_2$ has gonality at least 18, in part by computing the class group of this curve. As it was not clear to us how to implement these techniques in a fully generic and automated fashion, we have omitted them from our algorithm. Although such techniques may not benefit from the improved algorithm described above, an improvement in our ability to compute class groups would allow this technique to be applied further. This leads us nicely to the object of the next section.
	
	\section{Computing class groups} \label{sec:class_groups}
	
	In this section, we turn our attention to the problem of computing the divisor class group of the curve $X$. Our approach is based on the index calculus-type method of Hess presented in \cite{hess2003classgroups, hess1999phdthesis}, which we briefly summarize here. We set the following notation. Given a set $S$ of closed points of $X$, we say that a divisor $D \in \Div X$ is $S$-smooth if $D$ is supported on the points of $S$. The set of $S$-smooth divisors on $X$ is denoted $\Div_S X$, and forms a subgroup of $\Div X$ naturally isomorphic to the free abelian group $\Z^{|S|}$.
	
	The method of Hess proceeds in three main steps:
	\begin{description}
		\item[Factor basis construction] Determine a set of closed points $S \subset X$, called a factor basis, whose divisor classes generate the divisor class group of $X$. In particular, the map $\Z^{|S|} \cong \Div_S X \to \Cl X$ is surjective.
		\item[Relation collection] Denote by $\Lambda_S$ the kernel of the map $\Z^{|S|} \to \Cl X$. Compute a generating set $R$ of the lattice of relations $\Lambda_S$.
		\item[Group structure computation] Compute the quotient group $\Z^{|S|} / \Lambda_S$, as well as the isomorphism $\Cl X \cong \Z^{|S|} / \Lambda_S$.
	\end{description}
	
	In order to construct a factor basis, Hess provides an explicitly computable bound $m$ such that the set
	\[
		S = \{x \in X : \deg(x) \leq m\} \cup \Supp(D_1)
	\]
	generates the divisor class group $\Cl X$, where $D_1$ is any fixed divisor of degree 1 on $X$. This bound $m$ can be obtained using the \texttt{ClassGroupGenerationBound} intrinsic in Magma, while such a divisor $D_1$ can also be computed using an algorithm of Hess, implemented in Magma as the \texttt{DivisorOfDegreeOne} intrinsic. This procedure for constructing a factor basis is very efficient, and we proceed identically.
	
	The third step is also mostly straightforward. Given a generating set $R$ of the lattice of relations $\Lambda_S$, the quotient group $\Z^{|S|} / \Lambda_S$ can be computed from the Smith normal form of the matrix $B$ whose rows are the vectors of $R$. Indeed, we have that
	\[
		\Z^{|S|} / \Lambda_S \cong \frac{\Z}{n_1 \Z} \times \dots \times \frac{\Z}{n_{|S|} \Z},
	\]
	where $n_1, \dots, n_{|S|}$ are the diagonal entries of the Smith normal form of $B$, also known as the elementary divisors of $B$. In particular, the rank of $\Z^{|S|} / \Lambda_S$ is equal to $|S| - r$, where $r$ is the rank of $B$. Similarly, the isomorphism described above can be explicitly determined from the transformation matrices between $B$ and its Smith normal form. This allows us to compute the isomorphism $\Cl X \cong \Z^{|S|} / \Lambda_S$.
	
	As will be seen below, the matrix $B$ is quite sparse. As such, computing the elementary divisors of $B$ can be done reasonably efficiently using sparse techniques. The transformation matrices between $B$ and its Smith normal form however are always dense, and cannot be computed as efficiently. In practice, we have found that using the \texttt{AbelianGroup} functionality of Magma is quicker than computing the transformation matrices directly. However, this computation still scales much worse than computing the Smith normal form itself, and can dominate the running time in large examples. Work is ongoing on methods to circumvent this issue, however these lie outside the scope of the present article. Thus, in Section \ref{sec:timings}, we will omit the computation of the isomorphism $\Cl X \cong \Z^{|S|} / \Lambda_S$ when comparing methods, as any improvement to this procedure can be applied equally to both methods.
	
	The main difference between our method and the method of Hess lies in the relation collection stage. In the approach used by Hess, a generating set of the lattice of relations $\Lambda_S$ is computed as follows. Firstly, generate a random $S$-smooth (nearly) effective divisor $D \in \Div X$ of small degree, usually taken to be about the genus of $X$. Compute the Riemann-Roch space $\LL(D)$. For each function $f \in \LL(D)$, check whether the divisor $\div(f)$ is $S$-smooth. If so, $\div(f)$ is an element of the relation lattice $\Lambda_S$, and add the corresponding vector to the set $R$. This procedure is now repeated until $R$ forms a generating set of $\Lambda_S$.
	
	In order to check whether $R$ forms a generating set of $\Lambda_S$, Hess makes use of the known structure of the class group $\Cl X$. For instance, the class group $\Cl X$ is an abelian group of rank 1. Moreover, Hess provides an effective procedure for computing an approximation of the class number of $X$ with arbitrarily small multiplicative error, implemented in Magma in the \texttt{ClassNumberApproximation} intrinsic. In particular, one can compute an approximation $\tilde{h}$ of the class number with multiplicative error less than $\sqrt{2}$. From these facts, it follows that $R$ is a generating set of $\Lambda_S$ if and only if the quotient group $\Z^{|S|} / \langle R \rangle$ has rank 1 and the order of its torsion subgroup is smaller than $\sqrt{2} \tilde{h}$. Both the rank and the order of the torsion subgroup of this quotient group can be efficiently determined by computing the Smith normal form of the matrix $B$ as described above. This gives a practical procedure for checking whether the set $R$ generates the lattice $\Lambda_S$.
	
	The main bottleneck in this approach comes from the fact that one needs to compute the Riemann-Roch spaces $\LL(D)$ for a large number of $S$-smooth divisors $D \in \Div X$. Thus, just as was done in Section \ref{sec:gonalities}, our approach consists of replacing these Riemann-Roch computations with the linear algebraic alternative provided by Theorem \ref{thm:rr_matrix}. To do so, we generate the divisors $D$ in a different fashion. Instead of generating random small $S$-smooth effective divisors $D$, we first fix a large effective divisor $D_0 \in \Div X$, whose support we add to the factor basis $S$. We then compute the Riemann-Roch spaces $\LL(D_0 - D)$, where $D$ is a random $S$-smooth effective divisor such that the degree of $D_0 - D$ is small. This can be done efficiently using Theorem \ref{thm:rr_matrix} once the vectors $\mathbf{a}_{x, i, j}$ for a basis of $\LL(D_0)$ have been precomputed.
	
	Before describing the algorithm in detail, we first detail an additional improvement to the procedure which checks whether the divisor $\div(f)$ of a function $f \in k(X)$ is $S$-smooth. While this can be done by writing the divisor $\div(f)$ as a linear combination of closed points of $X$, for instance using the \texttt{Decomposition} intrinsic in Magma, this requires ideal factorization and is much too slow to be practical. Instead, we build on a method suggested by Hess in \cite{hess2003classgroups}, which relies on the following lemma.
	
	\begin{lemma}
		Consider a function $f \in k(X)$. Let $g \in k[t]$ be the denominator of $f$ with respect to the order $\OO_0$, that is, the minimal function $g \in k[t]$ such that $fg \in \OO_0$. Let $h \in k[t]$ be a generator of the ideal $k[t] \cap f g \OO_0$, and let $D_g$ and $D_h$ be the zero divisors of $g$ and $h$ respectively. Then
		\[
			\{\pi(x) : x \in \Supp(\div(f)), \pi(x) \neq \infty\} = \Supp(D_g) \cup \Supp(D_h).
		\]
		In particular, $\div(f)$ is supported on the set
		\[
			S = \{x \in X : \pi(x) \in \Supp(D_g) \cup \Supp(D_h) \cup \{\infty\}\}.
		\]
	\end{lemma}
	
	\begin{proof}
		Consider a closed point $x \in \Supp(\div(f))$ such that $\pi(x) \neq \infty$. Suppose first that $v_x(f) < 0$. Since $f g \in \OO_0$, it follows that $v_x(f) + v_x(g) \geq 0$, where the function $g \in k[t]$ is viewed as an element of $k(X)$. In particular, we have that
		\[
			v_x(g) \geq -v_x(f) > 0.
		\]
		Thus, viewing $g$ as an element of $k[t]$, we have that $v_{\pi(x)}(g) > 0$, and so $\pi(x) \in \Supp(D_g)$.
		
		Suppose now that $v_x(f) > 0$. Since $h \in k[t]$ is an element of the ideal $f g \OO_0$, there exists a function $a \in \OO_0$ such that $h = f g a$. As $g \in k[t] \subset \OO_0$ and $a \in \OO_0$, we obtain that
		\[
			v_x(h) = v_x(f) + v_x(g) + v_x(a) \geq v_x(f) > 0.
		\]
		Viewing $h$ as an element of $k[t]$, we therefore have that $v_{\pi(x)}(h) > 0$, and so $\pi(x) \in \Supp(D_h)$. Thus, we have shown that
		\[
			\{\pi(x) : x \in \Supp(\div(f)), \pi(x) \neq \infty\} \subset \Supp(D_g) \cup \Supp(D_h).
		\]
		
		Consider a closed point $y \in \Supp(D_g) \cup \Supp(D_h)$. Note that by definition, we have $y \neq \infty$. Moreover, there exists a function $b \in k(t)$ such that $\div(b) = y - \deg(y) \cdot \infty$. Suppose first that $y \in \Supp(D_g)$. By construction we have $v_y(g) > 0$, and so the quotient $\frac{g}{b}$ is an element of $k[t]$. The function $g \in k[t]$ is the denominator of $f$, and so we have $f g \in \OO_0$ and $f \frac{g}{b} \notin \OO_0$. As the divisor of $b$ is supported only on the closed points $y$ and $\infty$, it follows that there exists $x \in X$ with $\pi(x) = y$ such that $v_x(f \frac{g}{b}) < 0$. In particular, since $\frac{g}{b} \in k[t]$, we have
		\[
			v_x(f) \leq v_x \! \left(f \frac{g}{b}\right) < 0.
		\]
		Thus, $x \in \Supp(\div(f))$ and $y \in \{\pi(x) : x \in \Supp(\div(f)), \pi(x) \neq \infty\}$.
		
		Suppose now that $y \in \Supp(D_h) \setminus \Supp(D_g)$. Since $v_y(h) > 0$, the quotient $\frac{h}{b}$ is an element of $k[t]$. Since $h$ is a generator of the ideal $k[t] \cap f g \OO_0$, it follows that $h \in f g \OO_0$ and $\frac{h}{b} \notin f g \OO_0$. As before, write $h = f g a$ for some $a \in \OO_0$. Thus, we obtain that $\frac{a}{b} \notin \OO_0$. Since the divisor of $b$ is supported only on the closed points $y$ and $\infty$, it follows that there exists $x \in X$ with $\pi(x) = y$ such that $v_x(\frac{a}{b}) < 0$. As the quotient $\frac{h}{b}$ is an element of $k[t]$, we have $v_x(\frac{h}{b}) \geq 0$. Moreover, since $y \notin \Supp(D_g)$ and $g \in k[t]$, it follows that $v_y(g) = 0$. Thus, we have
		\[
			v_x(f) = v_x(f g) > v_x \! \left(f g \frac{a}{b}\right) = v_x \! \left(\frac{h}{b}\right) \geq 0.
		\]
		Thus, $x \in \Supp(\div(f))$ and $y \in \{\pi(x) : x \in \Supp(\div(f)), \pi(x) \neq \infty\}$. The statement of the lemma now follows.
	\end{proof}
	
	Using this result, we can give the following algorithm for computing whether the divisor $\div(f)$ of a function $f \in k(X)$ is $S$-smooth.
	
	\begin{algorithm} \label{alg:smoothness_check}
		(Smoothness check)
		
		Input: A factor basis $S = \{x_1, \dots, x_s\} \subset X$ and a function $f \in k(X)$ whose pole divisor is supported on $S$.
		
		Output: Whether the divisor of the function $f$ is supported on $S$, and if so, the vector in $\Z^s$ corresponding to $\div(f)$.
		
		\begin{enumerate}
			\item Compute the denominator $g \in k[t]$ of $f$ with respect to the order $\OO_0$.
			\item Compute a generator $h \in k[t]$ of the ideal $k[t] \cap f g \OO_0$.
			\item Compute the zero divisors $D_g$ and $D_h$ of $g$ and $h$ respectively.
			\item If $\Supp(D_h)$ contains a closed point $y \in \P^1$ such that no closed point of $S$ lies above $y$, return false and terminate. \vspace{0.5em}
			\item Set $d = 0$ and $\mathbf{v} = \mathbf{0} \in \Z^s$.
			\item For each closed point $x_i \in S$ such that $\pi(x_i) \in \Supp(D_g) \cup \Supp(D_h) \cup \{\infty\}$:
			\begin{enumerate}
				\item Compute the valuation $v_{x_i}(f)$.
				\item Set $d = d + v_{x_i}(f)$ and $\mathbf{v}_{i} = v_{x_i}(f)$.
			\end{enumerate}
			\item If $d = 0$, return true and $\mathbf{v}$. Otherwise, return false. Terminate.
		\end{enumerate}
	\end{algorithm}
	
	The denominator $g$ and generator $h$ can be readily computed in Magma using the \texttt{Minimum} intrinsic. Moreover, computing the zero divisors $D_g$ and $D_h$ is equivalent to factoring the polynomials $g$ and $h$ in $k[t]$. Very efficient polynomial-time algorithms for this problem are known and implemented in Magma.
	
	In order to carry out steps 4 and 6 efficiently, we precompute, for each closed point $x$ of the factor basis $S$, the closed point $y = \pi(x) \in \P^1$ over which $x$ lies. This is stored in a hash table, whose keys are the closed points of $\P^1$ and whose values are the list of closed points of $S$ lying above the corresponding closed point of $\P^1$. This allows one to retrieve the set of closed points of $S$ lying above a given closed point of $\P^1$ without needing to iterate through the entirety of $S$, making steps 4 and 6 much faster. This is especially important as this smoothness check is executed many times throughout the course of the relation collection stage.
	
	Using this algorithm for checking the smoothness of divisors of functions, we now describe the algorithm for collecting relations in detail.
	
	\begin{algorithm} \label{alg:relation_collection}
		(Relation collection)
		
		Input: A factor basis $S \subset X$, an $S$-smooth divisor $D_0$ and an approximation $\tilde{h}$ of the class number of $X$ with multiplicative error less than $\sqrt{2}$.
		
		Output: A generating set $R \subset \Z^s$ of the lattice of relations $\Lambda_S$.
		
		\begin{enumerate}
			\item Compute a $k$-basis $\{f_1, \dots, f_\ell\}$ of $\LL(D_0)$ using the algorithm of \cite{hess2002riemannroch}, where $\ell = \ell(D_0)$.
			\item For all closed points $x \in S$, compute the vectors $\mathbf{a}_{x, i, 0}$ associated to the function $f_i$ as defined in Theorem \ref{thm:rr_matrix} for $1 \leq i \leq \ell$ using Algorithm \ref{alg:function_expansions}.
			\item Set $R = \emptyset \subset \Z^s$. \vspace{0.5em}
			\item Select a random effective divisor $D$ consisting of a sum of distinct closed points of $S$, with $\deg(D) < \ell$ and maximal. If there are any closed points $x \in S$ such that no relation of $R$ is supported on $x$, ensure that at least one such point $x$ is in the support of $D$.
			\item Compute the subspace $\LL(D_0 - D) \subset \LL(D_0)$ using Theorem \ref{thm:rr_matrix}.
			\item For each function $f \in \P(\LL(D_0 - D))$:
			\begin{enumerate}
				\item Check whether the divisor of $f$ is supported on $S$ using Algorithm \ref{alg:smoothness_check}.
				\item If so, add the corresponding vector in $\Z^s$ to the set $R$ and go to step 8.
			\end{enumerate}
			\item If no such $S$-smooth function $f$ was found, go to step 4.
			\item Compute the rank of the matrix $B$ whose rows are the vectors of $R$. If $\rank(B) \neq s - 1$, go to step 4.
			\item Let $o$ be the product of the non-zero elementary divisors of the matrix $B$. If $\sqrt{2} \tilde{h} \leq o$, go to step 4. Otherwise, return $R$ and terminate.
		\end{enumerate}
	\end{algorithm}
	
	In practice, we do not execute steps 8 and 9 every time we find a relation, as computing the rank and elementary divisors of the matrix $B$ would then dominate the running time. Instead, we proceed in batches, collecting many relations of $R$ before executing steps 8 and 9.
	
	The matrix $B$ is very sparse, as each relation of $R$ is supported on at most $2 \deg(D_0)$ closed points. Therefore, we both store the vectors of $R$ in a sparse format, as well as utilize sparse methods for computing the rank and elementary divisors of the matrix $B$. This improves both the performance and memory usage of the algorithm.
	
	In step 4, we ensure that the divisor chosen includes a closed point $x \in S$ such that no relation of $R$ is supported on $x$, if such a closed point $x$ exists. This serves to increase the likelihood of the set of relations $R$ generating the lattice $\Lambda_S$, as each closed point of $S$ is guaranteed to belong to some relation of $R$.
	
	In the same step, we also look for effective divisors $D$ consisting of a sum of distinct closed points of $S$. This condition means that we only need to precompute the leading coefficient of the Laurent series expansion of the functions $f_i$ at each closed point of $S$, greatly speeding up the precomputation as explained after Algorithm \ref{alg:function_expansions}. However, this change does have theoretical implications for the algorithm, as it may no longer be possible to generate all of the relations of $\Lambda_S$ in this manner. In practice, however, this does not seem to be a concern.
	
	The choice of the base divisor $D_0$ is not directly specified, and has a great influence on the algorithm. For instance, if the base divisor is too small, it may be impossible to find a generating set of $\Lambda_S$ by looking at functions in the Riemann-Roch space $\LL(D_0)$, and the algorithm will then not terminate. In practice, we use the smallest multiple of the infinity divisor of degree at least $2 g(X) + d$, where $d$ is the maximal degree of a closed point of $S$. This choice stems from two main reasons. Firstly, using a multiple of the infinity divisor allows the Riemann-Roch space $\LL(D_0)$ to be computed very efficiently, by computing a \texttt{ShortBasis} of the zero divisor. Secondly, the degree choice seems to be a very ``safe'' choice, where the algorithm seems to terminate in most cases. However, this choice of divisor is not always optimal, and in many cases is larger than needed. Executing the algorithm with a smaller choice of divisor may then lead to a large decrease in the running time. To compensate for this, we allow the user to optionally specify a base divisor $D_0$, which will be used instead of the default choice explained above.
	
	\section{Timings} \label{sec:timings}
    
    In this section, we compare our methods with existing approaches to computing gonalities and class groups of curves over finite fields. To do so, we benchmark the various approaches on three different sets of curves, as follows:
    \begin{enumerate}
        \item Random curves of genus $3 \leq g \leq 10$, given as plane nodal curves, as generated by \texttt{RandomCurveByGenus}.
        \item Random curves of genus $11 \leq g \leq 13$, given as intersections of cubics in $\P^3$, as generated by \texttt{RandomCurveByGenus}.
        \item The reductions of the modular curves $X_0(N)$ modulo various primes of good reduction, for various $1 \leq N \leq 250$, given as canonical images by \texttt{ModularCurveQuotient}.
    \end{enumerate}
    These sets of curves were chosen to cover a wide variety of possible situations in which class group and gonality computations might be applied. This diverse choice allows us to illustrate how the efficiency of the various algorithms varies with the properties of the underlying curve.

    \subsection{Gonalities} \label{sec:gonality_timings}

    In the case of gonalities, we compare our algorithm, as described in Algorithm \ref{alg:has_function_new} and implemented in our \texttt{CurveArith} package, to the algorithm described in Algorithm \ref{alg:has_function_original}. The latter was used for instance by Najman and Orli\'c in \cite{najman2024gonality} to compute the gonalities of the modular curves $X_0(N)$.

    Our methodology is as follows. For the modular curves $X_0(N)$, we simply redo the work done in \cite{najman2024gonality} and compute lower bounds for the gonality of these curves over various finite fields. In particular, we apply the procedures described in Algorithms \ref{alg:has_function_original} and \ref{alg:has_function_new} with values of $N$, $q$ and $d$ as given in \cite[Proposition 5.15]{najman2024gonality} to reproduce their gonality lower bounds for $X_0(N)_{\F_q}$. Our results are given in Table \ref{tab:gonality_x0}.

    \begin{table}
        \centering
        \caption{Timing data for Algorithms \ref{alg:has_function_original} and \ref{alg:has_function_new} on the modular curves $X_0(N) / \F_q$ appearing in \cite[Proposition 5.15]{najman2024gonality}. The modular curve $X_0(277)$ has been omitted as we were unable to compute a model for it. All durations are in seconds.}
        \label{tab:gonality_x0}
        \begin{tabular}{N{3}N{1}N{4}N{5}|N{3}N{1}N{4}N{6}|N{3}N{1}N{5}N{7}}
        \toprule
        {$N$} & {$q$} & {\ref{alg:has_function_new}} & {\ref{alg:has_function_original}} & {$N$} & {$q$} & {\ref{alg:has_function_new}} & {\ref{alg:has_function_original}} & {$N$} & {$q$} & {\ref{alg:has_function_new}} & {\ref{alg:has_function_original}} \\ \midrule
        38 & 5 & {$<$}~ 1 & {$<$}~ 1 & 127 & 3 & 5 & 42 & 172 & 3 & 63 & 20060 \\
        44 & 5 & {$<$}~ 1 & {$<$}~ 1 & 128 & 3 & 1 & 7 & 175 & 2 & 7 & 31 \\
        53 & 7 & {$<$}~ 1 & {$<$}~ 1 & 130 & 3 & 172 & 19150 & 176 & 3 & 35 & 1812 \\
        61 & 3 & {$<$}~ 1 & {$<$}~ 1 & 132 & 5 & 282 & 216500 & 178 & 3 & 268 & 15600 \\
        76 & 5 & 8 & 75 & 134 & 3 & 28 & 1261 & 179 & 5 & 39 & 273 \\
        82 & 5 & 5 & 139 & 136 & 5 & 72 & 18330 & 180 & 7 & 210 & 1403000 \\
        84 & 5 & 8 & 350 & 137 & 3 & 2 & 8 & 181 & 3 & 5 & 103 \\
        86 & 3 & 2 & 21 & 140 & 3 & 142 & 108200 & 187 & 2 & 68 & 3816 \\
        93 & 5 & 6 & 74 & 144 & 5 & 2 & 110 & 189 & 2 & 68 & 161 \\
        99 & 5 & 5 & 42 & 147 & 5 & 23 & 344 & 192 & 5 & 135 & 292100 \\
        102 & 5 & 329 & 9906 & 148 & 5 & 159 & 87440 & 193 & 3 & 40 & 520 \\
        106 & 7 & 1110 & 70720 & 150 & 7 & 1276 & 145300 & 196 & 5 & 1056 & 71150 \\
        108 & 5 & 3 & 39 & 151 & 5 & 20 & 83 & 197 & 3 & 63 & 258 \\
        109 & 3 & {$<$}~ 1 & {$<$}~ 1 & 152 & 3 & 35 & 1567 & 198 & 5 & 22380 & 641500 \\
        112 & 3 & 2 & 14 & 153 & 5 & 210 & 1656 & 200 & 3 & 37 & 10070 \\
        113 & 3 & 1 & 6 & 154 & 5 & 3558 & 828000 & 201 & 2 & 188 & 12360 \\
        114 & 5 & 540 & 23900 & 157 & 3 & 19 & 568 & 217 & 2 & 152 & 249 \\
        115 & 3 & 5 & 37 & 160 & 7 & 1629 & 30340 & 229 & 3 & 84 & 2028 \\
        116 & 3 & 3 & 10 & 162 & 5 & 21 & 185 & 233 & 2 & 87 & 5900 \\
        117 & 5 & 4 & 19 & 163 & 5 & 149 & 3057 & 241 & 2 & 50 & 401 \\
        118 & 3 & 3 & 11 & 169 & 5 & 7 & 52 & 247 & 2 & 294 & 2996 \\
        122 & 3 & 5 & 14 & 170 & 3 & 1718 & 419600 &  &  &  &  \\ \bottomrule
        \end{tabular}
    \end{table}

    In the case of the random curves of genus $3 \leq g \leq 13$, an additional consideration was the fact that the runtime of both algorithms depends strongly on the gonality of the curve, as we terminate immediately upon finding a function of minimal degree. This variation would make it difficult to compare the results between different runs.

    Instead, we apply Algorithms \ref{alg:has_function_original} and \ref{alg:has_function_new} with a fixed degree $d = \lfloor \frac{g + 3}{2} \rfloor$, where $g$ is the genus of the random curve. Moreover, if a non-constant function is found in steps 3b and 6b of Algorithms \ref{alg:has_function_original} and \ref{alg:has_function_new} respectively, we do not terminate, and continue iterating through all effective divisors of degree $d$. In effect, we simply compute the Riemann-Roch dimensions of all effective divisors of degree $\lfloor \frac{g + 3}{2} \rfloor$. The value $\lfloor \frac{g + 3}{2} \rfloor$ was chosen as it is the gonality of a generic genus $g$ curve over an algebraically closed field \cite[Proposition A.1.v]{poonen2007gonality}, and as such, is most representative of the degrees encountered in practice. This procedure allows us to obtain an idea of the time it might take to compute the gonality of the curve, while being independent of the actual gonality of the curve.
    
    For each genus $g$ and base field $\F_q$, ten random curves were generated using the \texttt{RandomCurveByGenus} intrinsic. For each curve, the above procedure was then carried through using both Algorithms \ref{alg:has_function_original} and \ref{alg:has_function_new}. In the case of Algorithm \ref{alg:has_function_original}, in many cases it would have been prohibitively expensive to carry out this procedure to completion. In such cases, the procedure was carried out up to a fixed time bound, from which the total time was extrapolated. The results for the random curves of genus $3 \leq g \leq 10$ are given in Table \ref{tab:gonality_genus_3-10}, while the results for the random curves of genus $11 \leq g \leq 13$ are given in Table \ref{tab:gonality_genus_11-13}.

    \begin{table}
        \centering
        \caption{Timing data for Algorithms \ref{alg:has_function_original} and \ref{alg:has_function_new} on random curves of genus $3 \leq g \leq 10$ over various finite fields $\F_q$. The data is split based on the value of $n_1$ as defined in both algorithms, due to material differences in these two cases. The timing data for curves of odd genus $g$ is very similar to that of genus $g + 1$ and, in the interest of brevity, has been omitted. All durations are in seconds.}
        \label{tab:gonality_genus_3-10}
        \begin{tabular}{N{1}N{2}D{2.1}D{3.1}D{3.1}D{5.1}D{3.1}D{5.1}D{3.1}D{3.1}}
            \toprule
            &  & \multicolumn{2}{c}{$g = 4$} & \multicolumn{2}{c}{$g = 6$} & \multicolumn{2}{c}{$g = 8$} & \multicolumn{2}{c}{$g = 10$} \\ \cmidrule(lr){3-4} \cmidrule(lr){5-6} \cmidrule(lr){7-8} \cmidrule(lr){9-10}
            {$n_1$} & {$q$} & {\ref{alg:has_function_new}} & {\ref{alg:has_function_original}} & {\ref{alg:has_function_new}} & {\ref{alg:has_function_original}} & {\ref{alg:has_function_new}} & {\ref{alg:has_function_original}} & {\ref{alg:has_function_new}} & {\ref{alg:has_function_original}} \\ \midrule
            1 & 3 & 0.064 & 0.070 & 0.158 & 0.114 & 1.222 & 1.970 & 4.330 & 9.212 \\
            & 7 & 0.170 & 0.404 & 1.442 & 5.570 & 17.12 & 96.03 & 155.6 & 807.2 \\
            & 13 & 0.618 & 2.560 & 7.798 & 46.81 & 178.9 & 1673 & {-} & {-} \\
            & 23 & 1.376 & 11.51 & 43.55 & 330.2 & {-} & {-} & {-} & {-} \\
            & 31 & 2.814 & 27.89 & 112.8 & 995.1 & {-} & {-} & {-} & {-} \\
            & 59 & 11.10 & 194.7 & {-} & {-} & {-} & {-} & {-} & {-} \\
            & 89 & 28.90 & 667.0 & {-} & {-} & {-} & {-} & {-} & {-} \\ \midrule
            2 & 3 & 0.072 & 0.098 & 0.178 & 0.462 & 1.450 & 7.642 & 3.286 & 32.88 \\
            & 7 & 0.116 & 0.668 & 0.690 & 8.964 & 4.598 & 118.9 & 32.97 & 969.3 \\
            & 13 & 0.190 & 2.798 & 2.048 & 68.23 & 27.83 & 1774 & {-} & {-} \\
            & 23 & 0.454 & 10.89 & 6.748 & 377.1 & 163.9 & 18200 & {-} & {-} \\
            & 31 & 0.566 & 16.54 & 15.36 & 1167 & 572.6 & 93980 & {-} & {-} \\
            & 59 & 2.602 & 96.60 & 161.8 & 13980 & {-} & {-} & {-} & {-} \\
            & 89 & 7.746 & 329.3 & 645.6 & 63580 & {-} & {-} & {-} & {-} \\ \bottomrule
        \end{tabular}
    \end{table}

    \begin{table}
        \centering
        \caption{Timing data for Algorithms \ref{alg:has_function_original} and \ref{alg:has_function_new} on random curves of genus $11 \leq g \leq 13$ over various finite fields $\F_q$. All durations are in seconds.}
        \label{tab:gonality_genus_11-13}
        \begin{tabular}{N{1}N{1}D{3.1}D{5.1}D{3.1}D{5.1}D{3.1}D{5.1}}
            \toprule
            &  & \multicolumn{2}{c}{$g = 11$} & \multicolumn{2}{c}{$g = 12$} & \multicolumn{2}{c}{$g = 13$} \\ \cmidrule(lr){3-4} \cmidrule(lr){5-6} \cmidrule(lr){7-8}
            {$n_1$} & {$q$} & {\ref{alg:has_function_new}} & {\ref{alg:has_function_original}} & {\ref{alg:has_function_new}} & {\ref{alg:has_function_original}} & {\ref{alg:has_function_new}} & {\ref{alg:has_function_original}} \\ \midrule
            1 & 2 & 3.878 & 8.432 & 2.188 & 4.374 & 5.472 & 17.87 \\
            & 3 & 17.76 & 106.0 & 21.93 & 123.4 & 73.33 & 651.7 \\
            & 5 & 515.9 & 4730 & 548.1 & 4477 & {-} & {-} \\ \midrule
            2 & 2 & 3.604 & 25.23 & 4.386 & 26.02 & 6.568 & 64.08 \\
            & 3 & 12.78 & 250.2 & 16.36 & 695.0 & 36.49 & 1813 \\
            & 5 & 109.9 & 5509 & 128.6 & 6912 & 932.7 & 63740 \\
            & 7 & 580.9 & 86360 & 687.4 & 63970 & {-} & {-} \\ \bottomrule
        \end{tabular}
    \end{table}

    We make a few remarks about these results. Firstly, we note that Algorithm \ref{alg:has_function_new} performs much better than Algorithm \ref{alg:has_function_original} in almost all cases, except when the genus of the curve and the size of the base field are both very small. The improvement is most pronounced when the genus of the curve and the size of the base field are both large, in which case we can obtain a reduction in runtime of over two orders of magnitude. This points to the effectiveness of the linear algebraic approach to computing many Riemann-Roch spaces.

    In addition, we note that there is a big discrepancy in the runtime of Algorithm \ref{alg:has_function_new} based on the value of $n_1$, as defined in Lemma \ref{thm:rational_points_trick}. To explain this, in addition to measuring the total time taken by both approaches, we also measured the time taken by each step of the algorithm. This data is summarized in Table \ref{tab:gonality_genus_3-13_breakdown}.

    \begin{table}
        \centering
        \caption{Detailed timing data for some of the runs aggregated in Tables \ref{tab:gonality_genus_3-10} and \ref{tab:gonality_genus_11-13}. For each run, some properties of the curve are listed in the first three columns, followed by the number of closed points $x \in X$ for which series expansions were computed and the number of divisors $D \in \Div X$ whose Riemann-Roch dimensions were computed. The speeds of the various parts of each algorithm are given in the last three columns.}
        \label{tab:gonality_genus_3-13_breakdown}
        \begin{tabular}{N{2}N{2}N{1}N{5}N{6}D{3.1}N{5}D{4.1}}
            \toprule
            & & & & & \multicolumn{2}{c}{Algorithm \ref{alg:has_function_new}} & {Algorithm \ref{alg:has_function_original}} \\ \cmidrule(lr){6-7} \cmidrule(lr){8-8}
            {$g$} & {$q$} & {$n_1$} & {$x$} & {$D$} & {Expansions/s} & {RR/s} & {RR/s} \\ \midrule
            4 & 89 & 1 & 4126 & 464957 & 269.0 & 43490 & 659.5 \\
            4 & 89 & 2 & 97 & 156752 & 116.9 & 22520 & 465.3 \\ \addlinespace
            6 & 31 & 1 & 10421 & 454896 & 137.4 & 42160 & 485.3 \\
            6 & 31 & 2 & 513 & 358190 & 126.7 & 43000 & 375.0 \\ \addlinespace
            8 & 13 & 1 & 7903 & 247468 & 54.57 & 31690 & 150.0 \\
            8 & 13 & 2 & 868 & 208692 & 51.76 & 32310 & 132.4 \\ \addlinespace
            10 & 7 & 1 & 4045 & 86709 & 30.79 & 33870 & 97.01 \\
            10 & 7 & 2 & 735 & 72696 & 28.49 & 43270 & 96.60 \\ \addlinespace
            11 & 5 & 1 & 3478 & 80286 & 8.241 & 43400 & 15.17 \\
            11 & 5 & 2 & 814 & 124746 & 6.335 & 37460 & 9.197 \\ \addlinespace
            13 & 3 & 1 & 551 & 11055 & 7.936 & 40940 & 13.76 \\
            13 & 3 & 2 & 203 & 34567 & 5.235 & 33890 & 7.889 \\ \bottomrule
        \end{tabular}
    \end{table}

    We note that the time taken to precompute the series expansions at a closed point of $X$ is comparable to the time taken to compute the dimension of the Riemann-Roch space of a divisor on $X$ using the method of Hess \cite{hess2002riemannroch}. In fact, the former is around two to four times slower than the latter, and both become slower and slower as the genus of the underlying curve $X$ increases. This is to be expected, as both procedures make use of polynomial manipulations in the function field of $X$.

    On the other hand, once the series expansions have been computed, calculating the dimension of the Riemann-Roch space of a divisor on $X$ using Theorem \ref{thm:rr_matrix} is very fast and does not depend on the underlying curve. On our machine, we are consistently able to compute the dimensions of about 30000 to 40000 Riemann-Roch spaces per second, several orders of magnitude more than is possible using the method of Hess. This reflects the fact that each dimension computation is reduced to computing the dimension of a small matrix with entries in the finite coefficient field of $X$.

    Because of these two dynamics, the runtime of Algorithm \ref{alg:has_function_new} is heavily dependent on the difference between the number of closed points for which series expansions must be computed and the number of divisors whose dimensions must be computed. For curves of small genus over small finite fields, these two are comparable, and as such the runtime of Algorithm \ref{alg:has_function_new} is similar to that of Algorithm \ref{alg:has_function_original}. On the other hand, when the genus of the curve $X$ and the finite field $\F_q$ are large, the number of divisors grows much faster than the number of points. In these cases, the savings obtained by computing the dimensions of the Riemann-Roch spaces using linear algebra heavily outweigh the cost of precomputing the series expansions, and Algorithm \ref{alg:has_function_new} becomes much faster than Algorithm \ref{alg:has_function_original}.
    
    This also explains why Algorithm \ref{alg:has_function_new} is much faster on curves with many rational points. For such curves, the degrees of the closed points for which we need to compute the series expansions is smaller. As such, there are fewer such closed points, which, as explained above, has a large impact on the runtime of Algorithm \ref{alg:has_function_new}. This is particularly noticeable for the modular curves $X_0(140)$, $X_0(180)$ and $X_0(192)$, for which the value of $n_1$ is 3. In these cases, the speedup obtained by Algorithm \ref{alg:has_function_new} over Algorithm \ref{alg:has_function_original} exceeds three orders of magnitude.

    \begin{remark} \label{rmk:x0_198}
        We note that, in the computation of the gonalities of the modular curves $X_0(N)$ described in Table \ref{tab:gonality_x0}, the vast majority of the time is spent on the single modular curve $X_0(198)$. In fact, a large proportion of this time is spent computing power series expansions at a \emph{single} $\F_5$-rational point; at this single $\F_5$-rational point, Algorithm \ref{alg:differential_expansions} takes roughly 60 times longer than at the other $\F_5$-rational points of $X_0(198)$. This illustrates that the existing Magma intrinsics for computing residue class rings are particularly ill-suited to our algorithm at this point. A different approach, perhaps driven by a different representation of the function field or the closed points, may then yield sizable additional performance gains.
    \end{remark}
    
    \subsection{Class groups} \label{sec:classgroup_timings}

    In the case of class groups, we compare our implementation of the class group computation based around Algorithm \ref{alg:relation_collection} to the existing \texttt{ClassNumber} intrinsic in Magma. Our approach to benchmarking both algorithms is as follows. For each curve in one of the three sets described at the start of the section, we ran both our implementation and the \texttt{ClassNumber} intrinsic on the curve, with a timeout of one hour. Unlike as in Section \ref{sec:gonality_timings}, we were however unable to extrapolate the total runtime of the algorithms from their progress after an hour. The data for the random curves of genus $3 \leq g \leq 13$ is given in Table \ref{tab:classgroup_genus_3-13}, while a random subset of data for the modular curves $X_0(N)$ is given in Table \ref{tab:classgroup_x0}. The complete data for the latter can be found in the \texttt{timings} directory of our \texttt{CurveArith} repository. 

    \begin{table}
        \centering
        \caption{Timing data for Algorithms \ref{alg:relation_collection} and the \texttt{ClassNumber} intrinsic on random curves of genus $3 \leq g \leq 13$ over various finite fields $\F_q$. All durations are in seconds. A hyphen indicates that the computation timed out after one hour.}
        \label{tab:classgroup_genus_3-13}
        \begin{threeparttable}
            \begin{tabular}{N{2}D{1.1}D{3.1}D{3.1}D{4.1}D{4.1}D{4.1}D{4.1}D{2.1}}
                \toprule
                & \multicolumn{2}{c}{$g = 4$} & \multicolumn{2}{c}{$g = 7$} & \multicolumn{2}{c}{$g = 10$} & \multicolumn{2}{c}{$g = 13$} \\ \cmidrule(lr){2-3} \cmidrule(lr){4-5} \cmidrule(lr){6-7} \cmidrule(lr){8-9}
                {$q$} & {\ref{alg:relation_collection}} & {Magma} & {\ref{alg:relation_collection}} & {Magma} & {\ref{alg:relation_collection}} & {Magma} & {\ref{alg:relation_collection}} & {Magma} \\ \midrule
                2 & 0.156 & 0.11 & 0.644 & 0.768 & 3.294 & 3.486 & 17.48 & 26.35 \\
                5 & 0.268 & 0.498 & 2.804 & 8.714 & 10.13 & 92.33 & 212.4 & {-} \\
                13 & 0.906 & 1.15 & 3.884 & 54.64 & 77.78 & 3384 {\tnote{a}} & 3027 {\tnote{a}} & {-} \\
                31 & 3.95 & 9.964 & 26.31 & 2107.5 & 683.44 & {-} & {-} & {-} \\
                59 & 2.02 & 83.63 & 115.98 & 1093 & 2229 & {-} & {-} & {-} \\
                97 & 2.858 & 131.5 & 625.9 & 2365 {\tnote{a}} & {-} & {-} & {-} & {-} \\ \bottomrule
            \end{tabular}
            \footnotesize
            \begin{tablenotes}
                \item[a] Not all five of the runs terminated within an hour. The average has been taken substituting 3600 for the other runs, thus underestimating the true value.
            \end{tablenotes}
        \end{threeparttable}
    \end{table}

    \begin{table}
        \centering
        \caption{Timing data for Algorithms \ref{alg:relation_collection} and the \texttt{ClassNumber} intrinsic on a random subset of the modular curves $X_0(N)$ over various finite fields $\F_q$. All durations are in seconds. A hyphen indicates that the computation timed out after one hour.}
        \label{tab:classgroup_x0}
        \begin{tabular}{N{2}N{2}D{2.1}D{3.1}|N{2}N{2}D{4.1}D{4.1}|N{3}N{2}D{4.1}D{4.1}}
            \toprule
            {$N$} & {$q$} & \ref{alg:relation_collection} & {Magma} & $N$ & $q$ & \ref{alg:relation_collection} & {Magma} & $N$ & $q$ & \ref{alg:relation_collection} & {Magma} \\ \midrule
            30 & 31 & 0.51 & 0.35 & 67 & 17 & 2.62 & 6.25 & 98 & 59 & 112.54 & 2915.03 \\
            41 & 31 & 1.41 & 3.13 & 69 & 7 & 0.73 & 2.76 & 99 & 31 & 244.77 & {-} \\
            43 & 17 & 0.42 & 1.46 & 73 & 17 & 1.56 & 4.69 & 100 & 97 & 379.42 & 733.7 \\
            47 & 17 & 0.43 & 1.5 & 76 & 17 & 35.27 & 828.49 & 103 & 7 & 3.98 & 13.78 \\
            47 & 31 & 0.63 & 0.51 & 76 & 31 & 65.38 & 2172.58 & 103 & 17 & 16.71 & 358.82 \\
            51 & 31 & 4.83 & 20.34 & 77 & 59 & 96.39 & {-} & 110 & 7 & 575.09 & {-} \\
            52 & 17 & 1.29 & 3.45 & 78 & 17 & 991.86 & {-} & 112 & 59 & 1706.29 & {-} \\
            54 & 17 & 0.88 & 2.06 & 88 & 97 & 840.88 & {-} & 121 & 31 & 13.51 & 255.42 \\
            55 & 7 & 0.46 & 0.69 & 96 & 7 & 1.44 & 3.84 & 135 & 17 & 1374.07 & {-} \\
            55 & 59 & 20.05 & 171.51 & 97 & 59 & 199.62 & {-} & 139 & 31 & 701.78 & {-} \\ \bottomrule
        \end{tabular}
    \end{table}

    We note that, overall, Algorithm \ref{alg:relation_collection} performs better than the existing algorithm implemented in Magma. This improvement is most notable for curves of large genus over large finite fields, and, in certain cases, can approach two orders of magnitude. This dynamic is largely explained by the same phenomenon as for the gonality algorithms, wherein computing Riemann-Roch spaces using Theorem \ref{thm:rr_matrix} becomes faster only when the number of such spaces to compute is large relative to the number of series expansions which must be calculated. There is, however, significant variability in the obtained data, particularly for the modular curves $X_0(N)$. This, compounded with the timeout threshold, makes the precise quantification of the improvement obtained by Algorithm \ref{alg:relation_collection} difficult.

    However, one may notice that the improvement obtained in the case of class groups is, on the whole, smaller than that obtained for gonalities. This is due to a confluence of two factors. Firstly, the number of Riemann-Roch spaces which must be computed in Algorithm \ref{alg:relation_collection} tends to be much smaller than the number of Riemann-Roch spaces computed in Algorithm \ref{alg:has_function_new}. Indeed, the former only computes enough spaces to find a basis of the lattice of relations, while the latter exhaustively searches through all effective divisors of a given degree.
    
    Secondly, while the work done in Algorithm \ref{alg:has_function_new} almost exclusively consists of computing Riemann-Roch spaces, there are two main bottlenecks in Algorithm \ref{alg:relation_collection}. The first is, as expected, the computation of Riemann-Roch spaces in step 5. However, the second is the smoothness check of step 6a, which is carried out using Algorithm \ref{alg:smoothness_check}. This algorithm relies mainly on polynomial and ideal arithmetic and, as such, can be comparatively slow when run on millions of functions as in Algorithm \ref{alg:relation_collection}. As a result, the improvements brought about by Theorem \ref{thm:rr_matrix} can only serve to eliminate the first bottleneck, with the runtime of Algorithm \ref{alg:relation_collection} dominated in many examples by the second. Work is underway on an alternative version of Algorithm \ref{alg:smoothness_check} which relies solely on linear algebra, which may appear in a future version of our \texttt{CurveArith} package.
	
	\printbibliography
\end{document}